\documentclass[11pt]{amsart}
\usepackage{geometry}
\usepackage{graphicx, amsmath, amssymb}
\usepackage{epstopdf}
\usepackage[noadjust]{cite}
\usepackage{tikz}
\usetikzlibrary{calc}

\theoremstyle{plain}
\newtheorem{theorem}{Theorem}[section]
\newtheorem{corollary}[theorem]{Corollary}
\newtheorem{lemma}[theorem]{Lemma}
\newtheorem{proposition}[theorem]{Proposition}

\theoremstyle{definition}
\newtheorem{definition}[theorem]{Definition}
\newtheorem{example}[theorem]{Example}
\newtheorem{remark}[theorem]{Remark}

\newcommand{\mathq}{\mathbb{Q}}
\newcommand{\mathz}{\mathbb{Z}}
\newcommand{\mathr}{\mathbb{R}}

\newcommand{\ord}{\text{ord}}

\makeatletter

\renewcommand{\p@enumii}{}
\makeatother

\def\<#1>{\langle #1 \rangle}
\newcommand{\Ran}{\operatorname{Ran}}
\allowdisplaybreaks[1] 

\makeatletter
\def\@setkeywords{%
  {\itshape \keywordsname}:\enspace \@keywords\@addpunct.}
\def\@setsubjclass{%
  {\itshape\subjclassname}:\enspace\@subjclass\@addpunct.}
\makeatother

\newcommand{\acr}{\newline\indent}

\title{P-adic metric preserving functions and their analogues}
\author{Robert W. Vallin}
\address{Robert W. Vallin\acr Department of Mathematics, Lamar University, Beaumont, TX 77710 USA}
\email{robert.vallin@lamar.edu}

\author{Oleksiy A. Dovgoshey}
\address{Oleksiy A. Dovgoshey\acr Functions Theory Department, Institute of Applied Mathematics and Mechanics of NASU, Dobrovolskogo str. 1, Slovyansk 84100, Ukraine}
\email{oleksiy.dovgoshey@gmail.com}

\begin{document}
\begin{abstract}
The $p$-adic completion \(\mathbb{Q}_p\) of the rational numbers induces a different absolute value $|\cdot|_p$ than the typical $| \cdot |$ we have on the real numbers. In this paper we compare and contrast functions \(f \colon \mathbb{R}^{+} \to \mathbb{R}^{+}\), for which the composition with the \(p\)-adic metric \(d_p\) generated by \(|\cdot|_p\) is still a metric on $\mathq_p$, with the usual metric preserving functions and the functions that preserve the Euclidean metric on $\mathr$. In particular, it is shown that \(f \circ d_p\) is still an ultrametric on \(\mathbb{Q}_p\) if and only if there is a function \(g\) such that \(f \circ d_p = g \circ d_p\) and \(g \circ d\) is still an ultrametric for every ultrametric \(d\). Some general variants of the last statement are also proved.
\end{abstract}


\maketitle

\section{Introduction}

The \(p\)-adic numbers were first introduced in 1897 by K.~Hensel \cite{Hen1897JDM}\footnote{Hereinafter, the phrases such as ``It was first introduced in the year YYYY'' are abbreviations for ``As far as the authors known, it was first introduced in the year YYYY''.}. The basis for numerous applications of \(p\)-adic numbers and for algebraic studies of these numbers is the so-called \(p\)-adic valuation. The field of \(p\)-adic numbers \(\mathbb{Q}_p\) endowed with a metric \(d_p\) generated by \(p\)-adic valuation is also a fundamental example in the theory of ultrametric spaces. Nevertheless, many metric properties of the space \((\mathbb{Q}_p, d_p)\) remain unexplored now. This paper is an attempt to fill one of these gaps. The goal of the paper is to describe the structure of functions \(f\) for which the composition \(f \circ d_p\) remains an ultrametric or a metric on \(\mathbb{Q}_p\).

The paper is organized as follows. In Section~\ref{sec:2} we remained the concept of \(p\)-adic valuation and give some examples of calculating this quantity.

Section~\ref{sec:3} deals with metric preserving functions and ultrametric preserving ones with emphasis on functions which preserve the standard Euclidean metric on the real line.

Section~\ref{sec:4} contains our main results related to functions that preserve the \(p\)-adic metrics \(d_p\). Proposition~\ref{p1} describes the structure of all functions for which \(f \circ d_p\) is still a metric. Theorem~\ref{t3.7} explicitly establish a relationship between ultrametric preserving functions and functions \(f\) for which \(f \circ d_p\) is still an ultrametric. The characteristic properties of the last class functions are found in Proposition~\ref{p3.6}.

In Section~\ref{sec:5} using some concepts of order theory we generalize the known characterization of ultrametric preserving functions to the case of functions which preserve the ultrametrics of a given class of ultrametric spaces (Theorem~\ref{t5.6}) and find the necessary and sufficient conditions under which these functions behave similarly to functions preserving \(d_p\) (Theorem~\ref{t5.7}).

In what follows we will use the following designations
\begin{itemize}
\item \(\mathbb{P}\) --- the set of prime numbers \(p \geq 2\).
\item \(\mathbb{Q}\) --- the set (field) of rational numbers.
\item \(\mathbb{Q}_p\) --- the completion of \(\mathbb{Q}\) generated by \(p\)-adic valuation \(|\cdot|_p\).
\item \(\mathbb{R}\) --- the set (field) of real numbers.
\item \(\mathbb{R}^{+}\) --- the set (ordered set) of all non-negative real numbers.
\item \(\mathbb{Z}\) --- the set (ring) of all integer numbers.
\end{itemize}

\section{The $p$-adic Numbers}\label{sec:2}

Let $p$ be a prime number. We define a map, $|\cdot|_p$, on $\mathq$ by
$|0|_p = 0$ and
$$
|x|_p = p^{-\ord_px},
$$
where $\ord_px$, the order of $x$ with respect to $p$, is (a) the highest power of $p$ that divides $x$ if $x \in \mathz$ and (b) $\ord_pa - \ord_pb$ if $x = a/b$ where $a,b \in \mathz$, $b\ne 0$.

\begin{example}
For the number $25/18$, we have $|25/18|_2 = 2$, $|25/18|_3 = 9$, while $|25/18|_5 = 1/25$. For any other prime, $p$, $|25/18|_p = 1$. 
\end{example}

 We note that the range of the map $|\cdot|_p$ is the set $\{0\}\cup \{p^n\colon n \in \mathbb{Z} \}$ unlike the usual $|\cdot|$ on $\mathr$ whose values include all non-negative real numbers.

For fixed prime $p$ the map \(|\cdot|_p\) defines a norm on \(\mathbb{Q}\). Ostrowski's Theorem (see \cite{Koblitz}) implies every non-trivial absolute value on \(\mathbb{Q}\) is equivalent to either the usual absolute value $|\cdot |$ or a $p$-adic absolute value $|\cdot|_p$. The $p$-adic absolute value is a non-Archimedean norm as $|\cdot|_p$ has the property that
$$
|x + y|_p \le \max \{ |x|_p, |y|_p\}.
$$
This $p$-adic norm leads us to the \(p\)-adic metric on $\mathq$ defined by 
\[
d_p(x,y) = |x-y|_p.
\]
We actually have something stronger than a metric. Thanks to the non-Archimedean property \(d_p\) is an ultrametric. Rather than the ordinary Triangle Inequality, \(d_p\) satisfies the Strong Triangle Inequality
\[
d_p(x, y) \leq \max\{d_p(x, z), d_p(z, y)\}
\]
for all \(x\), \(y\), \(z \in \mathbb{Q}\). As shown below, Cauchy sequences do not behave the same way in with the $p$-adic norm as in the usual norm and $(\mathq, d_p)$ is not a complete metric space. There are many ways to prove the lack of completeness (see, for example, \cite{Katok} where a counting argument is presented).

The next example shows that sequences which are Cauchy in $\mathq$ with the standard Euclidean metric do not have to be Cauchy in $(\mathq, d_p)$. To prove this we use the following theorem.

\begin{theorem}
The sequence $\{a_n\}$ is Cauchy in $(\mathq, d_p)$ if only if
$$
\lim_{n\rightarrow \infty} |a_{n} - a_{n+1}|_p =0.
$$   
\end{theorem}

\begin{example}
The sequence \(\{a_n\}\), $a_n = 10^{-n}$ is {\em not} Cauchy in $(\mathq, d_p)$ for any prime $p$. If $p = 2,5$, $|10^{-n} - 10^{-m}|_p \rightarrow \infty$ while $n,m \rightarrow \infty$ and for any other $p$, $|a_n - a_{n+1}|_p = | 9 \cdot \frac{1}{10^{n+1}}|_p\ge 1/9$ hence not Cauchy by the theorem above.
\end{example}

\begin{example}\label{negone}
Fix $p \in \mathbb{P}$. The sequence whose terms are given by $a_n = \sum_{k=1}^n p^k$ is a Cauchy sequence in \((\mathbb{Q}, d_p)\). For $i < j$, we have
$$|a_j - a_i|_p = \left|\sum_{k=i+1}^j p^k \right|_p \le \max_{i+1\le k \le j} \{|p^k|_p\} = p^{-(i+1)}.$$
The limit is the number $L = 1/(1-p)$ since
	$$\left|  a_n - \frac{1}{1-p}  \right|_p =\left| \frac{1-p^{n+1}}{1-p} - \frac{1}{1-p} \right|_p = \left|  \frac{p^{n+1|}}{1-p} \right|_p = \frac{1}{p^{n+1}}.$$
\end{example}

With $p$ still fixed, we can complete $(\mathq, d_p)$ in the usual way.  Take the collection of all Cauchy sequences and place an equivalence relation on the set by $\{a_n\} \equiv \{b_n\}$ if $\lim_{n\rightarrow \infty}|a_n - b_n|_p = 0$. We make equivalence classes out of the collections of Cauchy sequences whose difference is a null sequence and the collection of equivalence classes forms the completion. This is the field of $p$-adic numbers, denoted $\mathq_p$.	

In \cite{Katok} see how to write every number in $\mathq_p$ uniquely as
$$
\cdots + d_4p^4+ d_3p^3+ d_2p^2 + d_1p + d_0 + \dfrac{d_{-1}}{p}+ \cdots + \dfrac{d_{-j+1}}{p^{j-1} } +  \dfrac{d_{-j}}{p^j}
$$
where $d_k \in \{0,1,2,\ldots,p-1\}$. Thus, every member of $\mathq_p$ can be written as
$$
\ldots d_kd_{k-1}\ldots d_1d_0.d_{-1}d_{-2} \ldots d_{-j}.
$$

\begin{example}
In the field of \(3\)-adic numbers $\mathq_3$
\begin{itemize}
\item The number seventeen is written as $122.$ 
\item The additive inverse of one is written as $ \ldots 2222222.$ We can show this using Example \ref{negone} with $p-1$ in place of $p$ in that example.
\item The multiplicative inverse of two is written as $ \ldots 1111112.$ Note that this says $1/2$ is a $3$-adic integer. 
\end{itemize}
\end{example}

\section{Metric Preserving Functions}\label{sec:3}

We now turn our attention to functions that preserve metrics.

\begin{definition}
A function $f: \mathr^+ \rightarrow \mathr^+$ is called metric preserving (ultrametric preserving) if for every metric (ultrametric) space $(X, \rho)$ the composition $f \circ \rho$ is still a metric (ultrametric) on~$X$.
\end{definition}

The canonical example of a metric preserving function is $f(x) = \frac{x}{1+x}$.

It is straightforward to see that one requirement is $f\colon \mathr^+ \rightarrow \mathr^+$ must be \emph{amenable}, i.e., \(f(0) = 0\) and \(f(x) > 0\) for every \(x > 0\).  For every metric space \((X, \rho)\) and any $x,y \in X$, the facts that $f(\rho (x,y)) \ge 0$ and  $f(\rho(x,y)) = f(\rho(y,x))$ are automatically satisfied. The other requirement is to make sure that the triangle inequality still holds up under composition with $f$. There are several ways to make sure this happens. 

\begin{example}
Let $f:\mathr^+ \rightarrow \mathr^+$ be amenable.
\begin{itemize}
\item If there exists an $a>0$ such that for $x>0$, $a \le f(x) \le 2a$, then $f$ is a metric preserving function. 
\item If $f$ is concave down, then $f$ is metric preserving.
\item If $a\le b+c$ implies $f(a) \le f(b) + f(c)$, then $f$ is metric preserving (Wilson's Theorem).
\end{itemize}
\end{example}

For general metric spaces, the structure of metric preserving functions has been firstly studied by W.~Wilson in 1935 \cite{Wil1935AJM} and systematically discussed by J.~Dobo\v{s} \cite{Dob1998}. Recent decades, the properties of metric preserving functions investigated by many mathematicians. In particular, the algebraic structure of the semigroup generated by metric preserving functions was firstly investigated by O.~Dovgoshey and O.~Martio in 2013~\cite{DM2013BazAuG}. The study of ultrametric preserving functions begun by P.~Pongsriiam and I.~Termwuttipong in 2014~\cite{PT}. 

\begin{remark}\label{r3.3}
The metric preserving functions can be considered as a special case of metric products (= metric preserving functions of several variables). See, for example, \cite{BD1981MS, BFS2003BazAuG, DPK2014MS, FS2002AG, HM1991CM}. It is interesting to note that an important special class of ultrametric preserving functions of two variables was first considered in 2009~\cite{DM2009RRMPA}.
\end{remark}

Theorem~19 of paper~\cite{PT} can be formulated as

\begin{lemma}\label{l1}
The following statements are equivalent for every \(\Psi \colon \mathbb{R}^{+} \to \mathbb{R}^{+}\):
\begin{enumerate}
\item [\((i)\)] We have \(\Psi(0) = 0\) and the double inequality 
\[
0 < \Psi(a) \leq 2\Psi(b)
\]
holds whenever \(0 < a < b\).
\item [\((ii)\)] The mapping \(X \times X \xrightarrow{d} \mathbb{R}^{+} \xrightarrow{\Psi} \mathbb{R}^{+}\) is a metric for every ultrametric space \((X, d)\).
\end{enumerate}
\end{lemma}

The dual to this lemma result was obtained in \cite{Dov2020MS}.

The following theorem, from \cite{PT} gives us an easy and complete description of ultrametric preserving functions.

\begin{theorem}\label{t2.4}
A function \(f \colon \mathbb{R}^{+} \to \mathbb{R}^{+}\) is ultrametric preserving if and only if \(f\) is increasing and amenable.
\end{theorem}

This theorem was generalized by O.~Dovgoshey in 2019 (see Proposition~3.24 and Corollary~3.25 \cite{Dov2019a}) to the special case of the so-called ultrametric distances. The ultrametric distances were first introduced by S.~Priess-Crampe and P.~Ribenboim in 1993 \cite{PR1993AMSUH} and studied in \cite{PR1996AMSUH, PR1997AMSUH, Rib1996PMH, Rib2009JoA}.

Also in  \cite{PT} one finds
\begin{corollary}\label{c2.5}
The following statements hold for every $f:\mathr^+\rightarrow \mathr^+$:
\begin{enumerate}
\item If $f$ is ultrametric preserving and subadditive, then $f$ is metric preserving.
\item If $f$ is metric preserving and increasing on $\mathbb{R}^{+}$, then $f$ is ultrametric preserving. 
\end{enumerate}
\end{corollary}

A way to see that a function is metric preserving, which we will make use of later on, involves the concept of a {\it triangle triplet} (see \cite{BD} and \cite{Corazza}). Let $(a,b,c)$ be a triple of three non-negative, real numbers. We say that \((a, b, c)\) is a triangle triplet if
$$
a \le b + c, \quad b \le a + c, \quad \text{and} \quad c \le a + b
$$
all hold.  We will write this as $(a,b,c) \in \Delta$. We say that $(a,b,c)$ is a \emph{strong triangle triplet}, $(a,b,c) \in \Delta_\infty$, if 
\[
a \le \max\{b, c\}, \quad b \le \max\{a, c\}, \quad \text{and}\quad c \le \max\{a, b\}
\]
(see \cite{PT}).

\begin{theorem}\label{t3.8}
Let $f: \mathr^+ \rightarrow \mathr^+$ be amenable. Then the following statements hold:
\begin{enumerate}
\item\label{t3.8:s1} $f$ is metric preserving if and only if $(f(a), f(b), f(c))$ is a triangle triplet whenever $(a,b,c)$ is one. 
\item\label{t3.8:s2} $f$ transfers ultrametrics to metrics if and only if $(f(a), f(b), f(c)) \in \Delta$ holds whenever $(a,b,c) \in \Delta_\infty$.
\item\label{t3.8:s3} $f$ is ultrametric preserving if and only if $(f(a), f(b), f(c)) \in \Delta_\infty$ holds whenever $(a,b,c) \in \Delta_\infty$.
\end{enumerate}
\end{theorem}

Statement~\ref{t3.8:s1} of Theorem~\ref{t3.8} is a reformulation of a corresponding result from~\cite{Corazza}. See also~\cite{PT} for Statements~\ref{t3.8:s2} and \ref{t3.8:s3} of this theorem.

An easy to prove consequence of the definition of triangle triplets is the following:
\begin{itemize}
\item [] Let $x_0 \in \mathr^{+}$ and let $f: \mathr^+ \rightarrow \mathr^+$ be metric preserving. If $f(x_0) = a$, then $f(x) \ge a/2$ for $x>x_0$. 
\end{itemize}
Other interesting properties of metric preserving functions include if $f$ is metric preserving and continuous at zero, then $f$ must be continuous everywhere.  

There are several pathological examples involving metric preserving function including a continuous, nowhere differentiable metric preserving function and a singular metric preserving function. In particular, since the (expanded) Cantor function is subadditive \cite{Dob1996PotAMS}, amenable, and increasing on~\(\mathbb{R}^{+}\), it is metric preserving by Theorem~\ref{t2.4} and Statement~\((i)\) of Corollary~\ref{c2.5}.

\medskip

If we focus on functions that preserve only the Euclidean metric on $\mathr$ things change. By restricting the space that is preserved we have an enlarged the set of functions we have on hand. 

Paper \cite{Dobos} shows that 
$$
f(x) = \sin (x) + \sin(\sqrt{2} \cdot x)
$$
is an example of function $f$ that preserves the Euclidean metric on \(\mathbb{R}^{+}\), but has the property that $\liminf_{x\rightarrow \infty} f(x) = 0$, hence is not metric preserving. Vallin, in \cite{Vallin}, expanded on this and created an $f$ that preserved the Euclidean metric on \(\mathbb{R}^{+}\), has $\liminf_{x\rightarrow \infty} f(x) = 0$, yet is unbounded, thus answering a question from Dobo\v{s} (``Must all functions that preserving the Euclidean metric and have $\liminf_{x\rightarrow \infty} f(x) = 0$ be almost periodic?'') in the negative.

\medskip

Triangle triplets can be used to test a function for preserving $(\mathr, |\cdot|)$ as follows:

\begin{theorem}
Let $f:\mathr^+ \rightarrow \mathr^+$. Then $f$ preserves the Euclidean metric on $\mathr$ if \(f\) is amenable and \((f(a), f(b), f(a+b))\) is a triangle triplet for all \(a\), \(b \ge 0\).
\end{theorem}

\section{Preserving $p$-adic Metrics}\label{sec:4}

Let us start from the basic for us

\begin{definition}\label{d3.1}
Let \(p \in \mathbb{P}\). A function \(f \colon \mathbb{R}^{+} \to \mathbb{R}^{+}\) is said to be \(p\)-adic metric (ultrametric) preserving if the composition \(f \circ d_p\) is a metric (an ultrametric) on \(\mathbb{Q}_p\).
\end{definition}

Given that the functions that preserve the Euclidean metric on \(\mathbb{R}\) have such different properties/examples than metric preserving functions. It seems logical to look at \(p\)-adic metric preserving functions. This paper does so with a collection of examples.

Having a discrete range makes a difference in preserving functions. Let us look at a straightforward example function defined on $\{0\} \cup \{p^n\}$  given by
\begin{center}
$f^*(0) = 0$, $f^*(p^n) = q^n$ for all integers $n$
\end{center}
where $p$, $q \in \mathbb{P}$ are fixed. We can see that $f^*\circ d_p$ changes things on $\mathq$ from the metric for $\mathq_p$ to the one for $\mathq_q$. One easy way to compose functions is to change the range from that of $|\cdot|_p$ to that of $|\cdot|_q$.

Since $\{p^n: p \in \mathbb{P} \text{ and } n \in \mathbb{Z}\}$ has only the origin as a limit point we can modify $f^*$ to $F^*$ that is  defined for all powers of all primes using $F^*(p^n) = q^n$. This function is only defined on a discrete space, so now we extend the function to be defined on all of $\mathr^+$ and continuous on the positive numbers. We can do this by requiring $F^*$ to be linear between consecutive positive points where it was previously defined. So now we have an amenable function $F^{*} \colon \mathr^+ \to \mathr^+$ that preserves all $p$-adic metrics but need {\em not} be a metric preserving function. This happens naturally if $q \ne 2$ as $(1/2^{n+1}, 1/2^{n+1}, 1/2^n)$ is a triangle triplet along with $c = a+b$, but for $q \ge 3$
$$
\frac{1}{q^{n+1}} + \frac{1}{q^{n+1}} = \frac{2}{q^{n+1}} < \frac{1}{q^n}
$$
since $2 < q$. 

This $F^*$ is not continuous at $0$. As the choice of prime $p$ increases, $1/p$ tends to zero while $F^*(1/p) = 1/q > 0$ stays constant. This example can be changed to a continuous function by ordering the primes $2=p_1<p_2<\cdots$ and having
$$
F(0) = 0, \, \,  F(p_k^n) = p_{k+1}^n, \text{ and $F$ linear elsewhere.}
$$
So we have a continuous $F$ that preserves $(\mathbb{Q}_p, d_p)$ for all $p \in \mathbb{P}$, but does not preserve the Euclidean metric on \(\mathbb{R}\). 

Our next example goes the other way; this function preserves the Euclidean metric on \(\mathbb{R}\), but not all \(p\)-adic metric. Specifically, our function $f$ will satisfy $(f(a), f(b), f(a+b))$ is a triangle triplet, but will not preserve the \(3\)-adic metric. 

\begin{example}\label{e3.1}
We define $f \colon \mathbb{R}^{+} \to \mathbb{R}^{+}$ piecewise as
$$
f(x) = \begin{cases}
x, & \text{if } 0 \le x \le 1 \\[7pt]
2-x, & \text{if } 1 \le x \le 3/2 \\[7pt]
\dfrac{3}{4}x - \dfrac{5}{8}, & \text{if } 3/2 \le x \le 2 \\[7pt]
-\dfrac{3}{4}x  + \dfrac{19}{8}, & \text{if } 2 \le x \le 3 \\[7pt]
\dfrac{3}{4} x - \dfrac{17}{8}, & \text{if } 3 \le x \le 7/2 \\[7pt]
\dfrac{1}{2}, & \text{otherwise}.
\end{cases}
$$
(See Figure~\ref{fig1}). This preserves the Euclidean metric on \(\mathbb{R}\) as it passes the $(a, b, a+b)$ to triangle triplets. The proof of this is similar to the proof used in \cite{Vallin}. The key is that the magnitude of the slopes are $1$ from $x = 0$ to $x = 3/2$ and then $3/4$, so smaller, and after $x = 7/2$ the $y$-values are constant at half the maximum so triangle triplets will be satisfied. However, $|1/2 - 1/3|_3 = 3$, $|1/3 - 1/4|_3 = 3$ and $|1/2 - 1/4|_3 = 1$, but $f(|1/2 - 1/3|_3) = 1/8$, $f(|1/3 - 1/4|_3) = 1/8$ and $f(|1/2 - 1/4|_3) = 1$ which will not satisfy the triangle inequality.
\end{example}

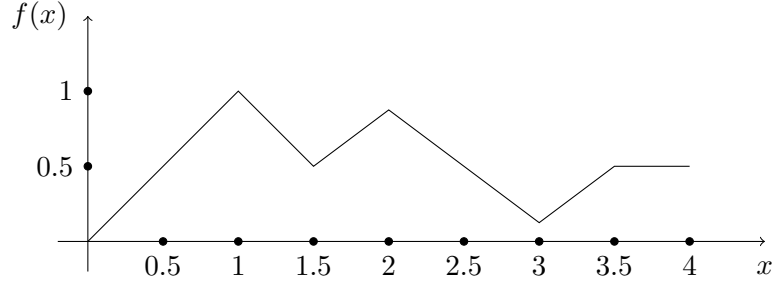
\begin{figure}
\begin{tikzpicture}
\def\xx{2cm}
\node [label=left:$f(x)$] at (0, 1.5*\xx) {};
\node [label=below:$x$] at (4.5*\xx, 0) {};
\node [draw, fill=black, circle, inner sep=1pt, label=left:$0.5$] at (0, 0.5*\xx) {};
\node [draw, fill=black, circle, inner sep=1pt, label=left:$1$] at (0, \xx) {};
\node [draw, fill=black, circle, inner sep=1pt, label=below:$0.5$] at (0.5*\xx, 0) {};
\node [draw, fill=black, circle, inner sep=1pt, label=below:$1$] at (1*\xx, 0) {};
\node [draw, fill=black, circle, inner sep=1pt, label=below:$1.5$] at (1.5*\xx, 0) {};
\node [draw, fill=black, circle, inner sep=1pt, label=below:$2$] at (2*\xx, 0) {};
\node [draw, fill=black, circle, inner sep=1pt, label=below:$2.5$] at (2.5*\xx, 0) {};
\node [draw, fill=black, circle, inner sep=1pt, label=below:$3$] at (3*\xx, 0) {};
\node [draw, fill=black, circle, inner sep=1pt, label=below:$3.5$] at (3.5*\xx, 0) {};
\node [draw, fill=black, circle, inner sep=1pt, label=below:$4$] at (4*\xx, 0) {};
\draw[->] (-0.2*\xx,0) -- (4.5*\xx,0);
\draw[->] (0,-0.2*\xx) -- (0, 1.5*\xx);
\draw (0,0) -- (\xx, \xx) -- (1.5*\xx, .5*\xx) -- (2*\xx, .875*\xx) -- (3*\xx, .125*\xx) -- (3.5*\xx, .5*\xx) -- (4*\xx, .5*\xx);
\end{tikzpicture}
\caption{The function \(f \colon \mathbb{R}^{+} \to \mathbb{R}^{+}\) preserves the Euclidean metric on \(\mathbb{R}\), but it is not \(3\)-adic metric preserving.}
\label{fig1}
\end{figure}

\emph{The \(p\)-adic metric preserving functions are different from metric preserving functions and from functions that preserve the Euclidean metric on \(\mathbb{R}\).}

\bigskip

Let us turn now to characteristic properties of \(p\)-adic metric (ultrametric) preserving functions.

\begin{lemma}\label{l3.2}
Let \(p \in \mathbb{P}\). Then for each pair \(n\), \(m \in \mathbb{Z}\) with \(n < m\) there are \(x\), \(y\), \(z \in \mathbb{Q}\) such that
\begin{equation}\label{l3.2:e1}
|x-z|_p = |z-y|_p = p^{m} \quad \text{and} \quad |x-y|_p = p^{n}.
\end{equation}
\end{lemma}
\begin{proof}
Let \(l \in \mathbb{Z}\). The numbers \(x\), \(y\), \(z \in \mathbb{Q}\) satisfy \eqref{l3.2:e1} if and only if 
\begin{equation}\label{l3.2:e3}
|p^{l}x-p^{l}z|_p = |p^{l}z-p^{l}y|_p = p^{m-l} \quad \text{and} \quad |p^{l}x-p^{l}y|_p = p^{n-l}.
\end{equation}
Consequently, it suffices to prove that for every strictly positive \(k \in \mathbb{Z}\) there are \(x\), \(y\), \(z \in \mathbb{Q}\) such that
\begin{equation}\label{l3.2:e2}
|x-z|_p = |z-y|_p = 1 \quad \text{and} \quad |x-y|_p = p^{k}.
\end{equation}
Indeed, \eqref{l3.2:e1} follows from \eqref{l3.2:e3} and \eqref{l3.2:e2} with \(l = m\) and \(k = m-n\). 

If \(p \geq 3\), then \eqref{l3.2:e2} holds with \(z = 1\), \(x = p^{k}\) and \(y = -p^{k}\). For \(p = 2\) and \(k \geq 2\) it suffices to set \(z = 1\), \(x = 2^{k-1}\) and \(y = -2^{k-1}\). Finally, for the case \(p = 2\) and \(k = 2\) we can set \(z = 0\), \(x = 1\) and \(y = -1\).
\end{proof}

\begin{proposition}\label{p1}
Let \(p \in \mathbb{P}\). Then the following conditions are equivalent for every function \(f \colon \mathbb{R}^{+} \to \mathbb{R}^{+}\):
\begin{enumerate}
\item [\((i)\)] \(f(0) = 0\) and the double inequality 
\[
0 < f(p^{m}) \leq 2f(p^{n})
\]
holds whenever \(m\), \(n \in \mathbb{Z}\) and \(m < n\).
\item [\((ii)\)] The function \(f\) is \(p\)-adic metric preserving.
\end{enumerate}
\end{proposition}
\begin{proof}
\((i) \Rightarrow (ii)\). Suppose \(f \colon \mathbb{R}^{+} \to \mathbb{R}^{+}\) satisfies condition \((i)\). Let us define a function \(\Psi_p \colon \mathbb{R}^{+} \to \mathbb{R}^{+}\) as
\begin{equation}\label{p1:e1}
\Psi_p(x) = \begin{cases}
0, & \text{if } x = 0\\
f(p^{m}), & \text{if } x \in [p^{m}, p^{m+1}) \text{ for } m \in \mathbb{Z}.
\end{cases}
\end{equation}
Then \(\Psi_p\) satisfies condition \((i)\) of Lemma~\ref{l1}. Consequently, \(\Psi_p \circ d_p\) is an ultrametric by Lemma~\ref{l1}. From~\eqref{l3.2:e1} it follows that \(f(|x-y|_p) = \Psi_p(|x-y|_p)\) for all \(x\), \(y \in \mathbb{Q}_p\). Hence, \(f \circ d_p\) is also an ultrametric.

\((ii) \Rightarrow (i)\). Suppose \((ii)\) holds. Then 
\[
f(|x-y|_p), f(|x-z|_p), f(|z-y|_p)
\]
is a triangle triplet. Using Lemma~\ref{l3.2}, for each pair \(n\), \(m \in \mathbb{Z}\) with \(n < m\), we can find \(x\), \(y\), \(z \in \mathbb{Q} \subseteq \mathbb{Q}_p\) such that \eqref{l3.2:e1} holds. Consequently, \((f(p^{m}), f(p^{m}), f(p^{n}))\) is a triangle triplet whenever \(n\), \(m \in \mathbb{Z}\) and \(n < m\). It implies the inequality \(f(p^{n}) \leq 2f(p^{m})\) whenever \(n < m\). Condition \((i)\) follows.
\end{proof}

\begin{remark}\label{r3.4}
Let \(p \in \mathbb{P}\) and let \(f \colon \mathbb{R}^{+} \to \mathbb{R}^{+}\) be amenable. If the inequality 
\[
f(p^{n}) \leq f(p^{n+1})
\]
holds for every \(n \in \mathbb{Z}\), then the function \(\Psi_p \colon \mathbb{R}^{+} \to \mathbb{R}^{+}\), defined by~\eqref{p1:e1}, is increasing and amenable. By Theorem~\ref{t2.4}, this is an ultrametric preserving function.
\end{remark}

The following example shows that \(f \circ d_2\) may not be a metric on \(\mathbb{Q}_2\) even if \(f \colon \mathbb{R}^{+} \to \mathbb{R}^{+}\) is amenable and the inequality 
\begin{equation}\label{e1}
f(2^{n-1}) \leq 2f(2^n)
\end{equation}
holds for every \(n \in \mathbb{Z}\).

\begin{example}
Let us define a function \(f \colon \mathbb{R}^{+} \to \mathbb{R}^{+}\) as
\[
f(x) = \begin{cases}
0, & \text{if } x = 0\\
\frac{1}{x}, & \text{if } x \neq 0.
\end{cases}
\]
Then \(f\) is amenable and
\[
f(2^n) = \left(\frac{1}{2}\right)^{n}
\]
holds for every \(n \in \mathbb{Z}\). Hence, we have \eqref{e1} for every \(n \in \mathbb{Z}\). Write \(x = 1\), \(y = -3\) and \(z = 0\). Then the equalities
\[
|x|_2 = 1, \quad |y|_2 = 1, \quad |x-y|_2 = 2^{-2}
\]
hold. It implies
\[
f(|x-z|_2) = f(|z-y|_2) = 1 \quad \text{and} \quad f(|x-y|_2) = 4.
\]
Since the \(3\)-tuple \((1, 1, 4)\) is not a triangle triplet, the function \(f\) is not \(2\)-adic metric preserving.
\end{example}

\begin{proposition}\label{p3.6}
Let \(p \in \mathbb{P}\). The following conditions are equivalent for every function \(f \colon \mathbb{R}^{+} \to \mathbb{R}^{+}\):
\begin{enumerate}
\item \(f(0) = 0\) and the double inequality \(0 < f(p^{n}) \leq f(p^{n+1})\) holds.
\item The function \(f\) is \(p\)-adic ultrametric preserving.
\end{enumerate}
\end{proposition}

\begin{proof}
\((i) \Rightarrow (ii)\). Let \((i)\) hold. In correspondence with Remark~\ref{r3.4}, the function \(\Psi_p \colon \mathbb{R}^{+} \to \mathbb{R}^{+}\), defined by~\eqref{p1:e1}, is ultrametric preserving. Consequently, for all \(x\), \(y\), \(z \in \mathbb{Q}_p\) the tuple 
\[
(\Psi_p(|x-y|_p), \Psi_p(|x-z|_p), \Psi_p(|z-y|_p))
\]
is a strong triangle triplet. Since, for every \(t \in \mathbb{Q}_p\), we have \(|t|_p \in \Ran(|\cdot|_p)\), where \(\Ran(|\cdot|_p)\) is the range of the function \(|\cdot|_p \colon \mathbb{Q}_p \to \mathbb{Q}_p\),
\[
\Ran(|\cdot|_p) = \{p^{n} \colon n \in \mathbb{Z}\} \cup \{0\},
\]
the equality \(\Psi_p(|t|_p) = f(|t|_p)\) holds for every \(t \in \mathbb{Q}_p\). Hence,
\[
(f(|x-y|_p), f(|x-z|_p), f(|z-y|_p))
\]
is also a strong triangle triplet. Statement \((ii)\) follows.

\((ii) \Rightarrow (i)\). Let \((ii)\) hold. Then \(f \circ d_p\) is an ultrametric on \(\mathbb{Q}_p\). Consequently, for every \(t \in \mathbb{Q}_p\), we have
\[
f(|t-t|_p) = f(0) = 0
\]
and, in addition, for \(n \in \mathbb{Z}\) and \(x = p^{-n}\),
\[
0 < f(|x - 0|_p) = f(|p^{-n}|_p) = f(p^{n}).
\]
By Lemma~\ref{l3.2}, there are \(x\), \(y\), \(z \in \mathbb{Q} \subseteq \mathbb{Q}_p\) such that
\[
|x-y|_p = p^{n-1} \quad \text{and} \quad |x-z|_p = |z-y|_p = p^{n}.
\]
Since \(f \circ d_p\) is an ultrametric, \((f(p^{n-1}), f(p^{n}), f(p^{n}))\) is a strong triangle triplet. Hence, 
\[
0 < f(p^{n-1}) \leq f(p^{n})
\]
holds. Statement \((i)\) follows.
\end{proof}

In the next theorem and everywhere in the future we will write \(f|_{A}\) for the restriction of the function \(f \colon \mathbb{R}^{+} \to \mathbb{R}^{+}\) to the set \(A \subseteq \mathbb{R}^{+}\).

\begin{theorem}\label{t3.7}
Let \(p \in \mathbb{P}\). Then the following statements are equivalent for every function \(f \colon \mathbb{R}^{+} \to \mathbb{R}^{+}\):
\begin{enumerate}
\item\label{t3.7:s1} There is an ultrametric preserving \(g \colon \mathbb{R}^{+} \to \mathbb{R}^{+}\) such that
\[
g|_{\Ran(|\cdot|_p)} = f|_{\Ran(|\cdot|_p)},
\]
where \(\Ran(|\cdot|_p)\) is the range of the function \(|\cdot|_p \colon \mathbb{Q}_p \to \mathbb{Q}_p\).
\item\label{t3.7:s2} The function \(f\) is \(p\)-adic ultrametric preserving.
\end{enumerate}
\end{theorem}

\begin{proof}
\(\ref{t3.7:s1} \Rightarrow \ref{t3.7:s2}\). Let \ref{t3.7:s1} hold. Then using Theorem~\ref{t2.4} and Proposition~\ref{p3.6} we obtain statement~\ref{t3.7:s2}.

\(\ref{t3.7:s2} \Rightarrow \ref{t3.7:s1}\). If \ref{t3.7:s2} holds, then \ref{t3.7:s1} follows from Theorem~\ref{t2.4}, Proposition~\ref{p3.6} and Remark~\ref{r3.4}.
\end{proof}

\section{Back to ultrametric preserving functions}\label{sec:5}

It is clear that for every \(p \in \mathbb{P}\) the range of the metric \(d_p \colon \mathbb{Q}_p \times \mathbb{Q}_p \to \mathbb{R}^{+}\) coincides with the range of the \(p\)-adic value \(|\cdot|_p\), \(\Ran(|\cdot|_p) = \Ran(d_p)\). Hence, by Theorem~\ref{t3.7}, a mapping \(f \circ d_p\) is a metric on \(\mathbb{Q}_p\) if and only if there exists an ultrametric preserving function \(g \colon \mathbb{R}^{+} \to \mathbb{R}^{+}\) such that
\[
f|_{\Ran(d_p)} = g|_{\Ran(d_p)}.
\]

The following is an example of an ultrametric space \((X, d)\) and a function \(f \colon \mathbb{R}^{+} \to \mathbb{R}^{+}\) for which \(f \circ d\) is an ultrametric on \(X\), but such that 
\begin{equation}\label{e5.1}
f|_{\Ran(d)} \neq g|_{\Ran(d)}
\end{equation}
for any ultrametric preserving function \(g \colon \mathbb{R}^{+} \to \mathbb{R}^{+}\).

\begin{example}\label{ex4.1}
Let \(X = \{x_1, x_2, x_3, x_4\}\) be a four-point set. A function \(d \colon X \times X \to \mathbb{R}^{+}\),
\[
d(x, y) = \begin{cases}
0, & \text{if } x = y\\
1, & \text{if } \{x, y\} = \{x_1, x_3\}\\
2, & \text{if } \{x, y\} = \{x_2, x_4\}\\
3, & \text{otherwise},
\end{cases}
\]
is an ultrametric on \(X\). We define a function \(f \colon \mathbb{R}^{+} \to \mathbb{R}^{+}\) as
\[
f(t) = \begin{cases}
2t, & \text{if } 0 \leq t \leq 1\\
3-t, & \text{if } 1 \leq t \leq 2\\
-1 + \frac{4}{3} t, & \text{if } 2 \leq t \leq 3\\
3, & \text{if } t \geq 3.
\end{cases}
\]
Then \(f\) is amenable, \(1 = f(2) < f(1) = 2\) holds and \(f \circ d\) is an ultrametric on \(X\). By Theorem~\ref{t2.4}, we obtain \eqref{e5.1} for all ultrametric preserving \(g \colon \mathbb{R}^{+} \to \mathbb{R}^{+}\). It is interesting to note that the ultrametric space \((X, d)\) and \((X, f \circ d)\) are isometric (see Figure~\ref{fig2}).
\end{example}

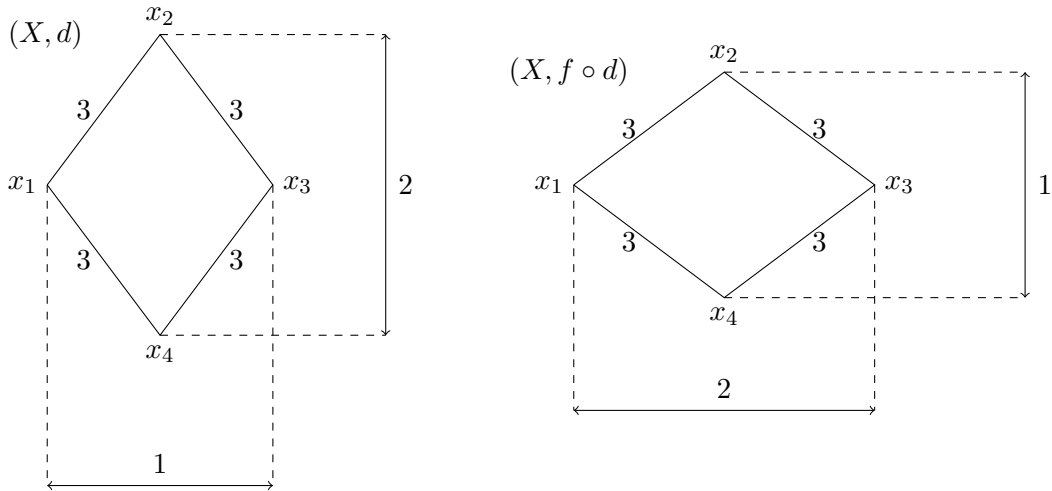
\begin{figure}[htb]
\begin{tikzpicture}
\def\xx{1.5cm}
\def\yy{2cm}
\coordinate [label=left:\(x_1\)] (x1) at (0,0);
\coordinate [label=above:\(x_2\)] (x2) at (\xx,\yy);
\coordinate [label=right:\(x_3\)] (x3) at (2*\xx,0);
\coordinate [label=below:\(x_4\)] (x4) at (\xx,-\yy);
\node [label=left:{$(X, d)$}] at (0.5*\xx,\yy) {};
\draw (x1) -- node [left=1pt] {\(3\)} (x2) -- node [right=1pt] {\(3\)} (x3) -- node [right=1pt] {\(3\)} (x4) -- node [left=1pt] {\(3\)} (x1);
\draw [dashed] (x1) -- ($(x1) - (0,2*\yy)$);
\draw [dashed] (x3) -- ($(x3) - (0,2*\yy)$);
\draw [<->] ($(x1) - (0,2*\yy)$) -- node [above=1pt] {\(1\)} ($(x3) - (0,2*\yy)$);

\draw [dashed] (x2) -- ($(x2) + (2*\xx,0)$);
\draw [dashed] (x4) -- ($(x4) + (2*\xx,0)$);
\draw [<->] ($(x2) + (2*\xx,0)$) -- node [right=1pt] {\(2\)} ($(x4) + (2*\xx,0)$);

\def\xx{2cm}
\def\yy{1.5cm}
\def\OO{7cm}
\coordinate [label=left:\(x_1\)] (x1) at (\OO,0);
\coordinate [label=above:\(x_2\)] (x2) at (\OO+\xx,\yy);
\coordinate [label=right:\(x_3\)] (x3) at (\OO+2*\xx,0);
\coordinate [label=below:\(x_4\)] (x4) at (\OO+\xx,-\yy);
\node [label=left:{$(X, f \circ d)$}] at (\OO+0.5*\xx,\yy) {};
\draw (x1) -- node [left=1pt] {\(3\)} (x2) -- node [right=1pt] {\(3\)} (x3) -- node [right=1pt] {\(3\)} (x4) -- node [left=1pt] {\(3\)} (x1);
\draw [dashed] (x1) -- ($(x1) - (0,2*\yy)$);
\draw [dashed] (x3) -- ($(x3) - (0,2*\yy)$);
\draw [<->] ($(x1) - (0,2*\yy)$) -- node [above=1pt] {\(2\)} ($(x3) - (0,2*\yy)$);

\draw [dashed] (x2) -- ($(x2) + (2*\xx,0)$);
\draw [dashed] (x4) -- ($(x4) + (2*\xx,0)$);
\draw [<->] ($(x2) + (2*\xx,0)$) -- node [right=1pt] {\(1\)} ($(x4) + (2*\xx,0)$);
\end{tikzpicture}
\caption{The mapping \(\Phi \colon X \to X\) with \(\Phi(x_i) = x_{i+1}\) for \(i = 1\), \(2\), \(3\) and \(\Phi(x_4) = x_1\) is an isometry of ultrametric spaces \((X, d)\) and \((X, f \circ d)\).}
\label{fig2}
\end{figure}

\begin{remark}\label{r4.2}
It can be proved that there are no isometric embeddings of the above ultrametric space \(X\) in the Euclidean plane \(E^{2}\), but we can find a four-point subset \(Y\) of the \(3\)-dimensional Euclidean space \(E^{3}\) such that \(Y\) and \(X\) are isometric (see Figure~\ref{fig3} and Lemma~\ref{l5.12} below).
\end{remark}

\begin{figure}[hbt]
\begin{tikzpicture}
\def\xx{1.5cm}
\def\yy{2cm}
\coordinate [label=left:\(y_1\)] (y1) at (0,0);
\coordinate [label=above:\(y_2\)] (y2) at (-\xx,\yy);
\coordinate [label=right:\(y_3\)] (y3) at (2*\xx,2*\yy);
\coordinate [label=below:\(y_4\)] (y4) at (2*\xx,\yy);

\node [label=above:{\(Y = \{y_1, y_2, y_3, y_4\}\)}] at (-\xx,1.5*\yy) {};

\draw (y1) -- node [left] {\(2\)} (y2) --node [above] {\(3\)} (y3) -- node [above] {\(3\)} (y1);
\draw (y2) -- node [below] {\(3\)} (y4) -- node [above] {\(3\)} (y1);
\end{tikzpicture}
\caption{The isosceles triangles \((y_1, y_2, y_3)\) and \((y_1, y_2, y_4)\) have the legs of length three and the common base \((y_1, y_2)\) of length two. The distance between the vertices \(y_3\) and \(y_4\) equals one.}
\label{fig3}
\end{figure}
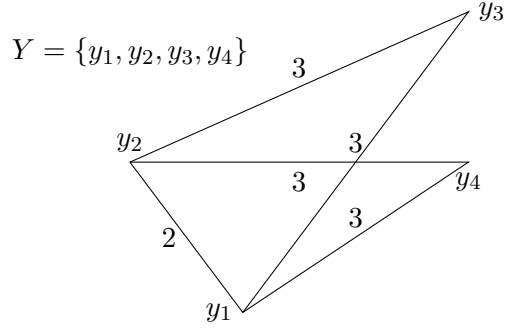

Now we want to describe all ultrametric spaces for which an analog of Theorem~\ref{t3.7} is valid. For this goal we recall some definitions.

Let \(X\) be a set. A \emph{binary relation} on \(X\) is a subset of the Cartesian square 
\[
X \times X = \{\<x, y> \colon x, y \in X\}.
\]

Let \(X\) be a set and let \(R_1\) and \(R_2\) be binary relations on \(X\). Recall that a \emph{composition of binary relations} \(R_1\) and \(R_2\) is a binary relation \(R_1 \circ R_2 \subseteq X \times X\) for which \(\<x, y> \in R_1 \circ R_2\) holds if and only if there is \(z \in X\) such that \(\<x, z> \in R_1\) and \(\<z, y> \in R_2\). 

A binary relation \(\alpha \subseteq X \times X\) is \emph{transitive} if
\[
(\<x,y> \in \alpha \text{ and } \<y, z> \in \alpha) \Rightarrow \<x,z> \in \alpha
\]
is valid for all \(x\), \(y\), \(z \in X\). Moreover, we say that \(\alpha\) is \emph{reflexive} if \(\<x,x> \in \alpha\) holds for every \(x \in X\).

Let \(\gamma\) be a binary relation on a set \(X\). We will write \(\gamma^{1} = \gamma\) and \(\gamma^{n+1} =  \gamma^{n}\circ \gamma\) for every integer \(n \geqslant 1\). The \emph{transitive closure} \(\gamma^{t}\) of \(\gamma\) is the relation
\begin{equation}\label{e3.2}
\gamma^{t} := \bigcup_{n=1}^{\infty} \gamma^{n}.
\end{equation}

For every \(\beta \subseteq X \times X\), the transitive closure \(\beta^{t}\) is transitive and the inclusion \(\beta \subseteq \beta^{t}\) holds. Moreover, if \(\tau \subseteq X \times X\) is an arbitrary transitive binary relation for which \(\beta \subseteq \tau\), then we also have \(\beta^{t} \subseteq \tau\), i.e., \(\beta^{t}\) is the smallest transitive binary relation containing \(\beta\).

Let us denote by \(\Delta_X\) the \emph{diagonal} of \(X\),
\[
\Delta_X = \{\<x, x> \colon x \in X\}.
\]
Then, for every \(\delta \subseteq X \times X\), the relation \(\delta \cup \Delta_X\) is the smallest reflexive binary relation on \(X\) containing the relation~\(\delta\).

\begin{definition}\label{d4.3}
Let \(U\) be a non-empty class of non-empty ultrametric spaces \((X, d)\) and let
\[
\Ran_U := \bigcup_{(X, d) \in U} \Ran(d).
\]
We denote by \(G_U\) the binary relation on the set \(\Ran_U\) defined by the rule: An ordered pair \(\<s, t>\) belongs to \(G_U\) if and only if there exist \((X, d) \in {U}\) and \(x_1\), \(x_2\), \(x_3 \in X\) such that
\begin{equation}\label{d4.3:e1}
s= d(x_1, x_3) \quad \text{and} \quad t = d(x_1, x_2) = d(x_2, x_3).
\end{equation}
\end{definition}

Recall that a reflexive and transitive binary relation \(\preccurlyeq_Y\) on a set \(Y\) is a \emph{partial order} on \(Y\) if, for all \(x\), \(y \in Y\), we have the \emph{antisymmetric property},
\[
\bigl(\<x, y> \in \preccurlyeq_Y \text{ and } \<y, x> \in \preccurlyeq_Y \bigr) \Rightarrow (x = y).
\]

In what follows we use the formula \(x \preccurlyeq y\) instead of \(\<x, y> \in \preccurlyeq\).

Let \(\preccurlyeq_Y\) be a partial order on a set \(Y\). A pair \((Y, {\preccurlyeq_Y})\) is called a \emph{poset} (a partially ordered set).

An element \(y_0 \in Y\) is, by definition, the \emph{least element} of poset \((Y, {\preccurlyeq_Y})\) if the inequality \(y_0 \preccurlyeq_Y y\) holds for every \(y \in Y\).

\begin{proposition}\label{p4.5}
Let \(U\) be a non-empty class of non-empty ultrametric spaces. The binary relation
\begin{equation}\label{p4.5:e1}
{\preccurlyeq_U} := G_U^{t} \cup \Delta_{\Ran_U}
\end{equation}
is a partial order on the set \(\Ran_U\). Moreover, the number \(0\) belongs to \(\Ran_U\) and this number is the least element of the poset \((\Ran_U, \preccurlyeq_U)\).
\end{proposition}

\begin{proof}
It follows directly from~\eqref{p4.5:e1} that \({\preccurlyeq_U}\) is reflexive and transitive. Suppose that \(s\) and \(t\) are distinct elements of \(\Ran_U\) such that \(s \preccurlyeq_U t\). Then \(\<s, t> \notin \Delta_{\Ran_U}\) holds and there is a positive integer \(n\) such that \(\<s, t> \in G_U^{n}\) (see Definition~\ref{d4.3}). Consequently, there are \((X_i, d_i) \in U\) and \(x_1^{i}\), \(x_2^{i}\), \(x_3^{i} \in X_i\), \(i = 1, \ldots, n\), such that
\begin{equation}\label{p4.5:e2}
\begin{split}
s &= d_1(x_1^{1}, x_3^{1}) \text{ and } d_1(x_1^{1}, x_2^{1}) = d_1(x_2^{1}, x_3^{1});\\
d_1(x_1^{1}, x_2^{1}) &= d_2(x_1^{2}, x_3^{2}) \text{ and } d_2(x_1^{2}, x_2^{2}) = d_2(x_2^{2}, x_3^{2})\\
\ldots & \ldots\\
d_{n-1}(x_1^{n-1}, x_2^{n-1}) &= d_{n}(x_1^{n}, x_3^{n}) \text{ and } d_{n}(x_1^{n}, x_2^{n}) = d_{n}(x_2^{n}, x_3^{n}) = t.
\end{split}
\end{equation}
Since all spaces \((X, d) \in U\) are ultrametric, from~\eqref{p4.5:e2} we obtain
\[
s \leq d_1(x_1^{1}, x_2^{1}) \leq d_2(x_1^{2}, x_2^{2}) \leq \ldots \leq d_{n}(x_1^{n}, x_2^{n}) = t.
\]
Hence, \(s \preccurlyeq_U t\) implies \(s \leq t\) if \(s \neq t\). We evidently also have \(\Delta_{\Ran_U} \subseteq {\leq}\). Consequently, the inclusion
\begin{equation}\label{p4.5:e3}
{\preccurlyeq_U} \subseteq {\leq}
\end{equation}
holds. Since \(\leq\) is antisymmetric, from \eqref{p4.5:e3} it follows \({\preccurlyeq_U}\) is also antisymmetric. Thus, \((\Ran_U, {\preccurlyeq_U})\) is a poset.

To complete the proof it suffices to show
\[
0 \in \Ran_U \text{ and } \<0, t> \in G_U
\]
for every \(t \in \Ran_U\).

Let \(t \in \Ran_U\). Then there are \((X, d) \in U\) and \(x\), \(y \in X\) satisfying \(t = d(x, y)\). The equality \(d(x, x) = 0\) implies \(0 \in \Ran_U\). It is clear that \eqref{d4.3:e1} holds with
\begin{equation}\label{p4.5:e4}
s = 0, \quad x_1 = x = x_3, \quad x_2 = y, \quad \text{and} \quad t = d(x, y).
\end{equation}
The statement \(\<0, t> \in G_U\) follows.
\end{proof}

Let \((X, {\preccurlyeq_X})\) and \((Y, {\preccurlyeq_Y})\) be posets. A mapping \(\Psi\colon X \to Y\) is \emph{isotone} if the implication
\[
(x \preccurlyeq_X y) \Rightarrow (\Psi(x) \preccurlyeq_Y \Psi(y))
\]
is valid for all \(x\), \(y \in X\). 

For posets \((X, {\preccurlyeq_X})\) and \((Y, {\preccurlyeq_Y})\) containing the least elements \(x_0 \in X\) and \(y_0 \in Y\), respectively, a mapping \(\Phi\colon X \to Y\) is \emph{amenable} if \(\Phi^{-1}(y_0) = \{x_0\}\) holds.

\begin{example}\label{ex5.5}
A function \(f \colon \mathbb{R}^{+} \to \mathbb{R}^{+}\) is amenable and increasing if and only if it is an amenable, isotone mapping from the poset \((\mathbb{R}^{+}, {\leq})\) to itself.
\end{example}

The next theorem can be considered as one of the main results of the paper.

\begin{theorem}\label{t5.6}
Let \(U\) be a non-empty class of non-empty ultrametric spaces \((X, d)\). Then the following conditions are equivalent for every \(f \colon \mathbb{R}^{+} \to \mathbb{R}^{+}\):
\begin{enumerate}
\item\label{t5.6:c1} The restriction \(f|_{\Ran_U}\) is an amenable, isotone mapping from the poset \((\Ran_U, {\preccurlyeq_U})\) in the poset \((\mathbb{R}^{+}, \leq)\).
\item\label{t5.6:c2} The function \(f \circ d\) is an ultrametric on \(X\) for every \((X, d) \in U\).
\end{enumerate}
\end{theorem}

\begin{proof}
\(\ref{t5.6:c1} \Rightarrow \ref{t5.6:c2}\). Let \ref{t5.6:c1} hold and let \((X, d) \in U\) be given. To prove that \(f \circ d\) is an ultrametric it suffices to show that the strong triangle inequality
\begin{equation}\label{t4.6:e1}
f(d(x, y)) \leq \max\{f(d(x,z)), f(d(z,y))\}
\end{equation}
holds for all \(x\), \(y\), \(z \in X\). Since \((X, d)\) is an ultrametric space, there is a permutation
\(
\left(
\begin{smallmatrix}
x & y & z \\
x_1 & x_2 & x_3
\end{smallmatrix}
\right)
\)
such that \(d(x_1, x_2) = d(x_2, x_3)\) holds. Consequently, the ordered pair \(\<s, t>\) with 
\begin{equation}\label{t4.6:e2}
s = d(x_1, x_3) \quad \text{and} \quad t = d(x_1, x_2) = d(x_2, x_3)
\end{equation}
belongs to binary relation \(G_U\) (see Definition~\ref{d4.3}). It was proved in Proposition~\ref{p4.5} that the binary relation \({\preccurlyeq_U} = G_U^{t} \cup \Delta_{\Ran_U}\) is a partial order on \(\Ran_U\). By condition~\ref{t5.6:c1}, the restriction \(f|_{\Ran_U}\) is an amenable and isotone mapping from \((\Ran_U, {\preccurlyeq_U})\) to \((\mathbb{R}^{+}, {\leq})\). Hence, we have the inequality
\[
f(s) \leq f(t).
\]
The last inequality and \eqref{t4.6:e2} imply
\begin{equation}\label{t4.6:e3}
f(d(x_1, x_3)) \leq f(d(x_1, x_2)) = f(d(x_2, x_3)).
\end{equation}
Since \((x_1, x_2, x_3)\) is a permutation of \((x, y, z)\), inequality~\eqref{t4.6:e1} follows from \eqref{t4.6:e3}.

\(\ref{t5.6:c2} \Rightarrow \ref{t5.6:c1}\). Let \ref{t5.6:c2} hold. We must prove that \(f|_{\Ran_U}\) is amenable and isotone as a mapping from \((\Ran_U, {\preccurlyeq_U})\) to \((\mathbb{R}^{+}, {\leq})\). By Proposition~\ref{p4.5}, the number \(0\) is the least element of \((\Ran_U, {\preccurlyeq_U})\). Since there is a non-empty space \((X, d)\), the equality \(d(x, x) = 0\) holds for some \(x \in X\). By condition \ref{t5.6:c2}, \(f \circ d\) is an ultrametric on \(X\). Hence, \(f(0) = f(d(x,x)) = 0\) holds. 

Let \(t\) be a non-zero element of \(\Ran_U\). Then there is \((Y, \rho) \in U\) such that \(\rho(y_1, y_2) = t\) for some distinct \(y_1\), \(y_2 \in Y\). Since \(f \circ \rho\) is an ultrametric on \(Y\), we have 
\[
f(t) = f(\rho(y_1, y_2)) \neq 0.
\]
Hence, \(f(t) \neq 0 = f(0)\) for every \(t \in \Ran_U\). Thus,
\begin{equation}\label{t4.6:e4}
f^{-1}(0) = \{0\}
\end{equation}
holds.

To complete the proof, it is enough to verify the truth of the implication
\begin{equation}\label{t4.6:e5}
(s \preccurlyeq_U t) \Rightarrow (f(s) \leq f(t))
\end{equation}
for all \(s\), \(t \in \Ran_U\). This is trivial if \(s = t\). For \(s = 0\) the validity of \eqref{t4.6:e5} follows from \eqref{t4.6:e4}. 

Let \(s \preccurlyeq_U t\) and \(s \neq t\) hold. Then, using \eqref{e3.2}, \eqref{d4.3:e1} and \eqref{p4.5:e1}, we can find a positive integer \(n\) such that \((X_i, d_i) \in U\) and \(x_1^{i}\), \(x_2^{i}\), \(x_3^{i} \in X_i\) for \(i = 1, \ldots, n\), and
\begin{align*}
s &= d_1(x_1^{1}, x_3^{1}) \text{ and } d_1(x_1^{1}, x_2^{1}) = d_1(x_2^{1}, x_3^{1});\\
d_1(x_1^{1}, x_2^{1}) &= d_2(x_1^{2}, x_3^{2}) \text{ and } d_2(x_1^{2}, x_2^{2}) = d_2(x_2^{2}, x_3^{2})\\
\ldots & \ldots\\
d_{n-1}(x_1^{n-1}, x_2^{n-1}) &= d_{n}(x_1^{n}, x_3^{n}) \text{ and } d_{n}(x_1^{n}, x_2^{n}) = d_{n}(x_2^{n}, x_3^{n}) = t
\end{align*}
(see~\eqref{p4.5:e2}). Consequently, we have
\begin{align*}
f(s) &= f(d_1(x_1^{1}, x_3^{1})) \text{ and } f(d_1(x_1^{1}, x_2^{1})) = f(d_1(x_2^{1}, x_3^{1}));\\
f(d_1(x_1^{1}, x_2^{1})) &= f(d_2(x_1^{2}, x_3^{2})) \text{ and } f(d_2(x_1^{2}, x_2^{2})) = f(d_2(x_2^{2}, x_3^{2}))\\
\ldots & \ldots\\
f(d_{n-1}(x_1^{n-1}, x_2^{n-1})) &= f(d_{n}(x_1^{n}, x_3^{n})) \text{ and } f(d_{n}(x_1^{n}, x_2^{n})) = f(d_{n}(x_2^{n}, x_3^{n})) = f(t).
\end{align*}
Since every \(f \circ d_i\), \(i = 1, \ldots, n\), is an ultrametric, we obtain
\[
f(s) \leq f(d_1(x_1^{1}, x_2^{1})) \leq f(d_2(x_1^{2}, x_2^{2})) \leq \ldots \leq f(d_{n}(x_1^{n}, x_2^{n})) = f(t).
\]
Hence, \(s \preccurlyeq_U t\) implies \(s \leq t\) if \(s \neq t\). Implication \eqref{t4.6:e5} is valid for all \(s\), \(t \in \Ran_U\).
\end{proof}

\begin{corollary}\label{c5.7}
Let \(U_1\) and \(U_2\) be non-empty classes of non-empty ultrametric spaces. Then the following conditions are equivalent:
\begin{enumerate}
\item The equalities 
\[
\Ran_{U_1} = \Ran_{U_2} \quad \text{and} \quad {\preccurlyeq_{U_1}} = {\preccurlyeq_{U_2}}
\]
hold.
\item For every \(f \colon \mathbb{R}^{+} \to \mathbb{R}^{+}\), a function \(f \circ d_1\) is an ultrametric on \(X_1\) for every \((X_1, d_1) \in U_1\) if and only if \(f \circ d_2\) is an ultrametric on \(X_2\) for every \((X_2, d_2) \in U_2\).
\end{enumerate}
\end{corollary}

Choosing various sub-classes \(U\) of the class of all metric spaces, it is easy to obtain a number of statements as a consequence of Theorem~\ref{t5.6}.

\begin{itemize}
\item Theorem~\ref{t2.4} simply follows from Theorem~\ref{t5.6} if \(U\) equals to the class of all non-empty ultrametric spaces.
\item Proposition~\ref{p3.6} follows from Theorem~\ref{t5.6} with \(U = \{(\mathbb{Q}_p, d_p)\}\), but we need Lemma~\ref{l3.2} for the proof.
\item Let \((X, d)\) be the ultrametric space from Example~\ref{ex4.1} and let \(U = \{(X, d)\}\). By Theorem~\ref{t5.6}, a mapping
\[
X \times X \xrightarrow{d} \mathbb{R}^{+} \xrightarrow{g} \mathbb{R}^{+}
\]
is an ultrametric on \(X\) if and only if we have
\[
0 < g(1) \leq g(3) \quad \text{and} \quad 0 < g(2) \leq g(3).
\]
\item Statement~\ref{t3.8:s3} of Theorem~\ref{t3.8} follows from Corollary~\ref{c5.7} with \(U_1\) equals to the class of all non-empty ultrametric spaces and \(U_2\) equals to the class of all three-point ultrametric spaces.
\item Let \(U\) be the class of all two-point ultrametric spaces. Then 
\[
X \times X \xrightarrow{d} \mathbb{R}^{+} \xrightarrow{g} \mathbb{R}^{+}
\]
is an ultrametric on \(X\) for every \((X, d) \in U\) if and only if \(g \colon \mathbb{R}^{+} \to \mathbb{R}^{+}\) is amenable.
\end{itemize}

An interesting modification of Theorem~\ref{t5.6} can be obtained for the case when the poset \((\Ran_U, {\preccurlyeq_U})\) is totally ordered.

Recall that a poset \((Y, {\preccurlyeq})\) is \emph{totally ordered} if, for all \(y_1\), \(y_2 \in Y\), we have
\[
y_1 \preccurlyeq y_2 \quad \text{or } \quad y_2 \preccurlyeq y_1.
\]

Let \(U\) be a class of ultrametric spaces. In what follows we use the notation
\[
\Ran_U^{0} := \Ran_U \setminus \{0\}, \quad T_U^{0} := \sup\Ran_U^{0} \quad \text{and} \quad t_U^{0} := \inf\Ran_U^{0}.
\]

\begin{theorem}\label{t5.7}
Let \(U\) be a non-empty class of ultrametric spaces such that \(|X| \geq 2\) holds for some \((X, d) \in U\). Then the following statements are equivalent:
\begin{enumerate}
\item\label{t5.7:c1} The poset \((\Ran_U, {\preccurlyeq_U})\) is totally ordered and the implications
\begin{equation}\label{t5.7:e0}
(T_U^{0} < \infty) \Rightarrow (T_U^{0} \in \Ran_U) \quad \text{and} \quad (0 < t_U^{0}) \Rightarrow (t_U^{0} \in \Ran_U)
\end{equation}
are valid.
\item\label{t5.7:c2} The following conditions are equivalent for every \(f \colon \mathbb{R}^{+} \to \mathbb{R}^{+}\):
\begin{enumerate}
\item\label{t5.7:c2.1} \(f \circ d\) is an ultrametric on \(X\) for every \((X, d) \in {U}\).
\item\label{t5.7:c2.2} There is an ultrametric preserving \(g \colon \mathbb{R}^{+} \to \mathbb{R}^{+}\) such that
\begin{equation}\label{t5.7:e1}
f|_{\Ran_U} = g|_{\Ran_U}.
\end{equation}
\end{enumerate}
\end{enumerate}
\end{theorem}

Now we state a lemma that will be used to prove Theorem~\ref{t5.7}.

\begin{lemma}\label{l5.8}
Let \((X, {\preccurlyeq_X})\) be a poset, let \(x_1\), \(x_2 \in X\) and \(p_1\), \(p_2 \in \mathbb{R}^{+}\) and let \(0 < p_1 < p_2\) hold. Write \([p_1, p_2] := \{x \in \mathbb{R}^{+} \colon p_1 \leq x \leq p_2\}\). If we have neither \(x_1 \preccurlyeq_X x_2\) nor \(x_2 \preccurlyeq_X x_1\), then there is an isotone mapping \(\Phi\) from \((X, {\preccurlyeq_X})\) to \(([p_1, p_2], {\leq})\) such that
\[
\Phi(x_1) = p_1 \quad \text{and} \quad \Phi(x_2) = p_2.
\]
\end{lemma}

This lemma is a very special case of the following

\begin{proposition}\label{p5.9}
Let \((P, {\leq_P})\) be a poset, \(A \subseteq P\), \({\leq_A} := ({\leq_P}) \cap (A \times A)\) and let \((L, {\leq_L})\) be a complete lattice. Then every isotone mapping \((A, {\leq_A}) \to (L, {\leq_L})\) has an extension to an isotone mapping \((P, {\leq_P}) \to (L, {\leq_L})\).
\end{proposition}

Various variants of Proposition~\ref{p5.9} are known \cite{BG2012AU, BNS2003PRSESA, Fof1969MN, Sik1948ASPM}. The formal proof of the above formulated version of this proposition can be found in~\cite{Dov2019UMV}.

\begin{proof}[Proof of Theorem~\ref{t5.7}.]
\(\ref{t5.7:c1} \Rightarrow \ref{t5.7:c2}\). Let \ref{t5.7:c1} hold. We must prove that \ref{t5.7:c2} holds. First of all, we note that the equality
\begin{equation}\label{t5.7:e2}
{\preccurlyeq_U} = (\leq \cap \Ran_U^{2})
\end{equation}
holds. Indeed, it was shown in the proof of Proposition~\ref{p4.5} that the inclusion \({\preccurlyeq_U} \subseteq {\leq}\) holds (see formula~\eqref{p4.5:e3}). Since \({\preccurlyeq_U}\) is an order on the set \(\Ran_U\), the last inclusion implies
\begin{equation}\label{t5.7:e3}
{\preccurlyeq_U} \subseteq ({\leq} \cap \Ran_U^{2}).
\end{equation}
We claim that the converse inclusion
\begin{equation}\label{t5.7:e4}
{\preccurlyeq_U} \supseteq ({\leq} \cap \Ran_U^{2})
\end{equation}
also holds. Indeed, if \(\<x, y> \in ({\leq} \cap \Ran_U^{2})\) and \(\<x, y> \notin {\preccurlyeq_U}\) for some \(x\), \(y \in \Ran_U\), then we have neither \(x \preccurlyeq_U y\) nor \(y \preccurlyeq_U x\) contrary to the condition that \((\Ran_U, {\preccurlyeq_U})\) is totally ordered. Equality \eqref{t5.7:e2} follows from \eqref{t5.7:e3} and \eqref{t5.7:e4}. 

If \eqref{t5.7:e1} holds with an ultrametric preserving \(g \colon \mathbb{R}^{+} \to \mathbb{R}^{+}\), then \(f \circ d\) is an ultrametric for every \((X, d) \in U\), because \(\Ran(d) \subseteq \Ran_U\). Consequently, to prove the validity of \ref{t5.7:c2} it suffices to show that \(f|_{\Ran_U}\) can be extended to an ultrametric preserving function \(f^{*} \colon \mathbb{R}^{+} \to \mathbb{R}^{+}\) if \ref{t5.7:c2.1} is valid.

By Theorem~\ref{t5.6}, the function \(f\) satisfies condition \ref{t5.7:c2.1} if and only if \(f|_{\Ran_U}\) is an amenable, isotone mapping from \((\Ran_U, {\preccurlyeq_U})\) to \((\mathbb{R}^{+}, {\leq})\). This statement and equality \eqref{t5.7:e2} imply that \(f|_{\Ran_U}\) is increasing and
\begin{equation}\label{t5.7:e5}
f^{-1}(0) \cap \Ran_U = \{0\}
\end{equation}
holds.

To construct an isotone extension of the mapping \(f|_{\Ran_U}\) to an amenable, increasing \(f^{*} \colon \mathbb{R}^{+} \to \mathbb{R}^{+}\) we consider the following four possible cases: \\
\((c_1)\) \(T_U^{0} = \infty\) and \(t_U^{0} = 0\); \((c_2)\) \(T_U^{0} = \infty\) and \(t_U^{0} > 0\); \((c_3)\) \(T_U^{0} < \infty\) and \(t_U^{0} = 0\); \((c_4)\) \(T_U^{0} < \infty\) and \(t_U^{0} > 0\).

Considering the empty set \(\varnothing\) as a subset of the poset \((\mathbb{R}^{+}, {\leq})\), one can obtain \(\sup\varnothing = 0\), but \(\inf \varnothing\) does not exist. Since \(|X| \geq 2\) holds for some \((X, d) \in U\), we have \(|\Ran_U^{0}| \geq 1\). Therefore, one of cases \((c_1)\)--\((c_4)\) always takes place.

Let us define \(f^{*} \colon \mathbb{R}^{+} \to \mathbb{R}^{+}\) as
\[
f^{*}(t) = \begin{cases}
0, & \text{if } t = 0\\
\sup\{f(x) \colon x \in [0, t] \cap \Ran_U\}, & \text{if } 0 < t < \infty
\end{cases}
\]
for case \((c_1)\);
\[
f^{*}(t) = \begin{cases}
0, & \text{if } t = 0\\
f(t_U^{0}), & \text{if } t \in (0, t_U^{0})\\
\sup\{f(x) \colon x \in [t_U^{0}, t] \cap \Ran_U\}, & \text{if } t \geq t_U^{0}
\end{cases}
\]
for case \((c_2)\);
\[
f^{*}(t) = \begin{cases}
0, & \text{if } t = 0\\
\sup\{f(x) \colon x \in [0, t] \cap \Ran_U\}, & \text{if } 0 < t < \infty
\end{cases}
\]
for case \((c_3)\);
\[
f^{*}(t) = \begin{cases}
0, & \text{if } t = 0\\
f(t_U^{0}), & \text{if } t \in (0, t_U^{0})\\
\sup\{f(x) \colon x \in [t_U^{0}, t] \cap \Ran_U\}, & \text{if } t_U^{0} \leq t < T_U^{0}\\
f(T_U^{0}), & \text{if } t \geq T_U^{0}
\end{cases}
\]
for case \((c_4)\).

Using condition~\ref{t5.7:c1} and equality~\eqref{t5.7:e5} it is easy to see that \(f^{*}\) is correctly defined and amenable. Moreover, form the definition of \(f^{*}\) it follows that \(f^{*}\) is increasing. Consequently, \(f^{*}\) is ultrametric preserving by Theorem~\ref{t2.4}.

From Theorem~\ref{t5.6} and equality~\eqref{t5.7:e2} it follows that \(f(t_1) \leq f(t_2)\) holds whenever \(t_1\), \(t_2 \in \Ran_U\) and \(t_1 \leq t_2\). This fact and the definition of \(f^{*}\) imply \(f(t) = f^{*}(t)\) for every \(t \in \Ran_U\). Thus, equality~\eqref{t5.7:e1} holds with \(g = f^{*}\). Statement \ref{t5.7:c2} follows.

\(\ref{t5.7:c2} \Rightarrow \ref{t5.7:c1}\). Let \ref{t5.7:c2} hold. If the poset \((\Ran_U, {\preccurlyeq_U})\) is not totally ordered, then we can find \(x_1\), \(x_2 \in \Ran_U\) such that we have neither
\begin{equation}\label{t5.7:e6}
x_1 \preccurlyeq_U x_2 \quad \text{nor} \quad x_2 \preccurlyeq_U x_1.
\end{equation}
Since \(\Delta_{\Ran_U} \subseteq {\preccurlyeq_U}\) and \(0\) is the least element of \((\Ran_U, {\preccurlyeq_U})\), condition~\eqref{t5.7:e6} implies
\[
x_1 \neq x_2 \neq 0 \neq x_1.
\]
Let \(p_1\), \(p_2\) belong to \(\mathbb{R}^{+}\) and 
\begin{equation}\label{t5.7:e7}
0 < p_1 < p_2
\end{equation}
hold. Let \({\preccurlyeq_U^{0}}\) be the restriction of \({\preccurlyeq_U}\) on the set \(\Ran_U^{0}\). By Lemma~\ref{l5.8}, there is an isotone mapping \(f^{0}\) from the poset \((\Ran_U^{0}, {\preccurlyeq_U^{0}})\) to the poset \(([p_1, p_2], {\leq})\) such that
\begin{equation}\label{t5.7:e8}
f^{0}(x_1) = p_1 \quad \text{and} \quad f^{0}(x_2) = p_2.
\end{equation}
Without loss of generality one can assume that \(x_1 > x_2\) because condition \eqref{t5.7:e6} is symmetric in \(x_1\) and \(x_2\) and \(x_1 \neq x_2\). Let us define now a mapping \(f \colon \Ran_U \to \mathbb{R}^{+}\) by the rule
\begin{equation}\label{t5.7:e9}
f(x) = \begin{cases}
0, & \text{if } x = 0\\
f^{0}(x), & \text{if } x \in \Ran_U^{0}.
\end{cases}
\end{equation}
The point \(0\) is the least element of \((\Ran_U, {\preccurlyeq_U})\) and, at the same time, the least element of \((\mathbb{R}^{+}, {\leq})\). The last statement and the isotonicity of \(f^{0}\) imply that \(\Ran_U \xrightarrow{f} \mathbb{R}^{+}\) is also isotone. Moreover, from~\eqref{t5.7:e9} and \eqref{t5.7:e7} it follows directly that \(f\) is amenable. By Theorem~\ref{t5.6}, the function \(f \circ d\) is an ultrametric on \(X\) for every \((X, d) \in U\). Since~\ref{t5.7:c2} holds, there is an ultrametric preserving \(g \colon \mathbb{R}^{+} \to \mathbb{R}^{+}\) such that \(f|_{\Ran_U} = g|_{\Ran_U}\). Consequently, we have \(x_1 > x_2\) and
\[
g(x_1) = f(x_1) = p_1 < p_2 = f(x_2) = g(x_2)
\]
contrary to Theorem~\ref{t2.4}. Thus, \ref{t5.7:c2} implies that \((\Ran_U, {\preccurlyeq_U})\) is totally ordered.

Suppose now that the inequality \(T_U^{0} < \infty\) holds but \(T_U^{0} \notin \Ran_U\). Since \(U\) contains an ultrametric space \((X, d)\) such that \(|X| \geq 2\), we have \(T_U^{0} > 0\). Let us define a function \(f\) on the interval \([0, T_U^{0})\) as
\begin{equation}\label{t5.7:e10}
f(x) = \begin{cases}
0, & \text{if } x = 0\\
\frac{1}{T_U^{0} - x}, & \text{if }0 < x < T_U^{0}.
\end{cases}
\end{equation}
Then the restriction \(f|_{\Ran_U}\) is an amenable, isotone mapping from \((\Ran_U, {\preccurlyeq_U})\) to \((\mathbb{R}^{+}, {\leq})\). Using Theorem~\ref{t2.4}, Theorem~\ref{t5.6} and condition~\ref{t5.7:c2.1} we can find an increasing \(g \colon \mathbb{R}^{+} \to \mathbb{R}^{+}\) such that \eqref{t5.7:e1} holds. Consequently, we have 
\begin{equation}\label{t5.7:e11}
f(x) \leq g(T_U^{0})
\end{equation}
for every \(x \in \Ran_U\). From \(T_U^{0} \notin \Ran_U\) and \(T_U^{0} = \sup\Ran_U^{0}\) it follows that there is a sequence \(\{x_n\}\), \(x_n \in \Ran_U^{0}\) such that \(\lim_{n\to \infty} x_n = T_U^{0}\). Equalities~\eqref{t5.7:e10} and \eqref{t5.7:e11} give us the contradiction,
\[
\infty = \lim_{n\to \infty} x_n \leq g(T_U^{0}) < \infty.
\]
Hence, the first implication from~\eqref{t5.7:e0} is valid. The validity of the second implication from~\eqref{t5.7:e0} can be proved similarly. Statement~\ref{t5.7:c1} is proved.
\end{proof}

To formulate and prove the last result of the paper we recall a known fact from the theory of ultrametric spaces.

Let \((X, d)\) and \((Y, \rho)\) be metric spaces. The space \((X, d)\) is isometrically embeddable in \((Y, \rho)\) if there is a map \(\Phi \colon X \to Y\) such that
\[
d(x, y) = \rho(\Phi(x), \Phi(y))
\]
holds for all \(x\), \(y \in X\).

\begin{lemma}\label{l5.12}
Let \(n \geq 1\) be an integer number. A finite ultrametric space \((X, d)\) is isometrically embeddable in the \(n\)-dimen\-sional Euclidean space \(E^{n}\) if and only if \(|X| \leq n+1\) holds.
\end{lemma}

This was independently proved by A.~F.~Timan \cite{Tim1975PSIM} (at least in a special case), by A.~Yu.~Le\-min \cite{Lem1985SMDa}, and by M.~Aschbacher, P.~Baldi, E.~B.~Baum, R. M. Wilson (see Theorems~1.1 and 6.7 in~\cite{ABBW1987SJADM}).

\begin{proposition}\label{p5.13}
Let \((X, d)\) be an ultrametric space with \(|X| \geq 2\). If \((X, d)\) is isometrically embeddable in the Euclidean plane \(E^{2}\), then statements~\ref{t5.7:c1} and \ref{t5.7:c2} of Theorem~\ref{t5.7} are valid with \(U = \{(X, d)\}\).
\end{proposition}

It follows directly from Lemma~\ref{l5.12}.

\begin{remark}
Example~\ref{ex4.1} shows that we cannot replace the plane \(E^{2}\) by other Euclidean space \(E^{n}\) with \(n \geq 3\) in Proposition~\ref{p5.13}.
\end{remark}

It would be interesting to find analogues of Theorem~\ref{t5.6} and Theorem~\ref{t5.7} for functions which preserve the metrics from an arbitrary fixed class.


\end{document}